\documentclass{amsart}

\usepackage{amssymb,amsbsy}
\usepackage{amsrefs}
\newtheorem{theorem}{Theorem}[section]
\newtheorem{lemma}[theorem]{Lemma}
\newtheorem{proposition}[theorem]{Proposition}
\newtheorem{corollary}[theorem]{Corollary}
\newtheorem{example}[theorem]{Example}
\numberwithin{equation}{section}

\newcommand{\ring}{S}
\newcommand{\coef}{\boldsymbol{\ring}}
\newcommand{\Ffield}{\mathbb{F}}
\newcommand{\field}{\mathbb{K}}
\newcommand{\fl}{\boldsymbol{Q}}
\newcommand{\fh}{\mathfrak h}
\newcommand{\J}{\mathcal J}
\newcommand{\Jg}{{\mathcal J}^g}
\newcommand{\N}{\mathbb N}
\newcommand{\s}{\mathfrak s}
\newcommand{\supp}{\mathop{\mathrm{supp}}}
\newcommand{\Z}{\mathbb Z}

\newcommand{\abf}{\mathbf{a}}

\DeclareMathOperator{\alg}{\mathrm{alg}}
\DeclareMathOperator{\Endo}{\mathrm{End}}
\DeclareMathOperator{\Lie}{\mathcal{L}}
\DeclareMathOperator{\St}{\mathrm{St}}

\font\cyrillic=wncyi10
\newcommand{\SU}{\mathop{\hbox{{\cyrillic UX}}}}

\author{J. M. P\'erez-Izquierdo}
\address{Departamento de Matem\'aticas y Computaci\'on, Universidad de
La Rioja, 26004 \\ Lo\-gro\-\~no, Spain}
\email{jm.perez@unirioja.es}
\dedicatory{Dedicated with admiration, respect and affection to Ivan Shestakov \\ on the occasion of his 70th birthday \\Thanks for making your home our mathematical home}

\begin{document}
\title[Formal power series with noncommutative coefficients]{The loop of formal power series with noncommutative coefficients under substitution}
\begin{abstract}
The set of formal power series with coefficients in an associative but noncommutative  algebra becomes a loop with the substitution product. We initiate the study of this loop by describing certain Lie and Sabinin algebras related to it. Some examples of Lie algebras satisfying the standard identities of degrees $5$ and $6$ appear naturally.
\end{abstract}
\maketitle

%
\section{Introduction}

Loops are the nonassociative counterpart of groups. A loop $(Q, xy, e)$ is a set $Q$ with a binary product $xy$ and a unit element $e \in Q$, i.e. $ex = x = xe$ for all $x \in Q$, such that the left and right multiplication operators $L_x \colon y \mapsto xy$, $R_x \colon y \mapsto yx$ by $x$ are invertible for all $x \in Q$. The left and right divisions on $Q$ defined by $ x \backslash y := L^{-1}_x(y)$ and $x/y := R^{-1}_y(x)$ satisfy
\begin{equation}\label{eq:loop}
	x\backslash (xy) = y = x(x\backslash y), \quad  (xy)/y = x = (x/y)y \quad \text{and} \quad x\backslash x = y/  y
\end{equation}
and loops can be introduced as those algebraic structures $(Q,xy, x\backslash y, x/y)$ for which (\ref{eq:loop}) holds. In case that $Q$ is a manifold topological, differentiable or analytical and the maps $xy, x/y, x\backslash y$ have the corresponding properties of regularity then the loop is called topological, differentiable or analytical. If the product is only defined on a neighborhood $U$ of $e$, i.e. $U \times U \rightarrow Q$, then we have a local loop. 

Groups are associative loops. Local Lie groups are classified in terms of their Lie tangent algebras, and local analytic loops are classified in terms of their Sabinin tangent algebras. Many authors contributed to pave the way to this result. Malcev in \cite{Ma55} realized that Lie's fundamental theorems could be extended to some varieties of loops. In the next few years Hofmann systematically studied topological loops. In the sixties Kikkawa \cite{Ki64} showed how  local geodesic sums in affine manifolds define local loops. In the early seventies Kuzmin \cite{Ku71} was able to integrate Malcev algebras to obtain local analytic Moufang loops, i.e. local analytic loops that satisfy any of the Moufang identities
\begin{align*}
	x(y(xz)) &= ((xy)x)z, & (xy)(zx) &= (x(yz))x,\\
	(xy)(zx) &= x((yz)x), & ((xy)z)y &= x(y(zy)),
\end{align*}
a major contribution later extended to global Moufang loops by Kerdman \cite{Ke79}. The work on local analytic Moufang loops was also extended to local analytic Bol loops, i.e. local analytic loops that satisfy
\begin{displaymath}
	((xy)z)y = x((yz)y),
\end{displaymath} 
by Mikheev and Sabinin in \cite{MiSa82}. Then, in \cite{HoSt86} Hofmann and Strambach made a fundamental contribution. They proved that the tangent algebra of any local analytic loop is an Akivis algebra and that these algebras can be integrated to obtain local analytic loops. Unfortunately, the same Akivis algebra might correspond to nonisomorphic local analytic loops, so a finer algebraic structure was required for classification. This new algebraic structure, called Sabinin algebra, came from Mikheev and Sabinin in \cite{MiSa87}. Thus, Sabinin algebras are the nonassociative counterpart of Lie algebras. Many other important names should have been mentioned, so we refer to \cites{Pf90,HoSt90,CPS90,Sa99,Sa00,NaSt02,AkGo06} for a historical account  of the different schools that contributed to the impressive advances of the subject during these decades. On the other side, the advances in finite loops during these years have been even more spectacular. 

In the approach of Mikheev and Sabinin to the study of local loops, right monoalternative loops play a central role. These are loops that satisfy the identity
\begin{displaymath}
(xy)y = x(yy).
\end{displaymath}
It was proved that any local analytic loop is the perturbation of an associated right monoalternative one. Thus Mikheev and Sabinin first described the tangent algebras of right monoalternative loops and then they included the extra operations required to codify the perturbation. 

A local loop $(Q,xy,e)$ determines the parallel transport of an affine connection with zero curvature by \[\tau^x_y(\xi_x) := dL_y\vert_e (dL_x\vert_e)^{-1}(\xi_x)\] where $df\vert_e$ denotes the differential of $f$ at $e$. Conversely, such an affine connection determines on $Q$ a right monoalternative local loop structure by \[x \times y := \exp_x \tau^e_x \exp^{-1}_e(y).\] Thus $xy = x \times \Phi_x(y)$ for some $\Phi_x$ that can be thought as the map required to perturb the monoalternative product $x \times y$ to recover the original product $xy$.

The tangent space $T_e Q$ of $Q$ at $e$ inherits many multilinear operations, the Sabinin brackets on $T_e Q$, from the torsion tensor by
\begin{displaymath}
\langle \xi_1,\dots, \xi_n; \zeta,\eta \rangle := \nabla_{\xi_1^*}\cdots \nabla_{\xi_n^*}\vert_e T(\zeta^*,\eta^*)
\end{displaymath}
where $\zeta^*$ indicates the adapted vector field associated to the vector $\zeta \in T_eQ$ and $\nabla$ denotes the covariant derivative. The relations among these operations are governed by Bianchi identities. The axioms for the Sabinin brackets that these identities provide are:
\begin{gather*}
	\langle x_1,\dots, x_m;y,z \rangle = - \langle x_1,\dots, x_m; z,y \rangle \\ \displaybreak[0] 
	\langle x_1,\dots, x_r,a,b,x_{r+1} \dots,x_m; y,z \rangle - \langle x_1,\dots, x_r,b,a,x_{r+1},\dots, x_m;y,z \rangle \\ 
	+\sum_{k=0}^r\sum_\alpha \langle x_{\alpha_1},\dots, x_{\alpha_k},\langle x_{\alpha_{k+1}},\dots, x_{\alpha_r};a,b\rangle, \dots, x_m; y,z \rangle = 0\\
	\sigma_{x,y,z}\left(\langle x_1,\dots, x_r,x;y,z\rangle + \sum_{k=0}^r \sum_{\alpha} \langle x_{\alpha_1},\dots, x_{\alpha_k};\langle x_{\alpha_{k+1}},\dots, x_{\alpha_r};y,z\rangle, x \rangle \right) = 0
\end{gather*}
where $m \geq 0$, $\alpha$ runs the set of all $(k,r-k)$ shuffles, i.e.  bijections of the type $\alpha \colon \{1,2,\dots,r\} \rightarrow \{1,2,\dots,r\}$, $i \mapsto \alpha_i$, $\alpha_1 < \alpha_2 <\cdots \alpha_k$, $\alpha_{k+1} < \cdots < \alpha_r$, $k=0,1,\dots,r$, $r\geq0$,  and $\sigma_{x,y,z}$ denotes the cyclic sum on $x,y,z$. The perturbation $\Phi_x$ is recovered by another family of multilinear operations $\Phi(x_1,\dots,x_m; y_1,\dots, y_n)$ ($m\geq 1$, $n\geq 2$), the multioperator, subject to the following axiom
\begin{displaymath}
	\Phi(x_1,\dots,x_m; y_1,\dots, y_n) = \Phi(x_{\tau(1)},\dots, x_{\tau(m)}; y_{\delta(1)},\dots, y_{\delta(n)})
\end{displaymath}
for any $\tau \in S_m, \delta \in S_n$, where $S_l$ is the symmetric group on $l$ symbols. A Sabinin algebra $(\s,\langle - ; - , - \rangle, \Phi(-;-))$ is a vector space equipped with Sabinin brackets and a multioperator.  Mikheev and Sabinin proved that two local loops are isomorphic if and only if their Sabinin tangent algebras are isomorphic.

In this paper we focus on the computation of the Sabinin brackets for the loop of formal power series with noncommutative coefficients under substitution. For a unital commutative ring $R$, the set $\J(R) \subseteq R[[t]]$ of formal power series $\sum_{i \geq 0} \alpha_i t^{i+1}$ with coefficients in $R$ and $\alpha_ 0 = 1$ is a group with the substitution product
\begin{equation}\label{eq:Nottingham}
	\left(\sum_{i \geq 0 } \alpha_i t^{i+1} \right) \circ \left(\sum_{j \geq 0 } \beta_j t^{j+1} \right) := \sum_{i \geq 0} \alpha_i \left(\sum_{j \geq 0} \beta_j t^{j+1}\right)^{i+1}.
\end{equation}	
This group was introduced by Jennings in \cite{Je54}. Later, Johnson \cite{Jo88}  and York \cite{Yo90} called the attention of group theorists leading to an intensive study of $\J(\Ffield_q)$ for finite fields $\Ffield_q$ (here $q=p^e$ for some prime $p$), i.e. the Nottingham groups. Looking at $\J(\Ffield_p)$ as a subgroup of the group of automorphisms of $\Ffield_p((t))$, Leedham-Green and Weiss proved that $\J(\Ffield_p)$ contains a copy of every finite $p$-group. Through a detailed study of Witt's result on Galois extensions used by Leedham-Green and Weiss, and by means of a result from Lubotzky and Wilson that ensures that there exists a $2$-generated pro-p group in which all countably based pro-p groups can be embedded, in \cite{Ca97} Camina proved that every countably based pro-p group can be embedded as a closed subgroup in $\J(\Ffield_p)$, a property that shares with the Grigorchuk group. Surprisingly enough, no explicit elements of order $p^2$ or embeddings of $\Z_p \times \Z_p$ have been obtained until recently \cite{BBK16}. We refer to the surveys \cites{Ca00,Ba13} for further information.

When the algebra of coefficients $R$ is no longer commutative then (\ref{eq:Nottingham}) might not be associative. Thus, instead of a group we get a loop, and our aim is to describe certain Sabinin brackets associated to it.  We will also present some results on normal subloops of $\J(R)$ in the spirit of \cite{Kl00}. Also,  as a consequence, we will obtain new examples of Sabinin algebras. Loops $\J(R)$ have appeared  in connection with Hopf algebras related to the renormalization procedure in quantum field theory \cite{BrFrKr06}, so a better understanding of them is desirable.

While related, the Sabinin brackets we will compute arise from a filtration of $\J(R)$ similar to the commutator-associator filtration introduced by Mostovoy in \cites{Mo06,Mo08} rather than from the geometric context of affine connections. Starting with the commutator 
\begin{displaymath}
	[a,b] = (b \circ a) \backslash (a \circ b)
\end{displaymath}
and the associator 
\begin{displaymath}
	(a,b,c) = (a \circ (b \circ c)) \backslash ((a \circ b) \circ c)
\end{displaymath} 
Mostovoy recursively  introduced the deviations of the associator by
\begin{displaymath}
	(a_1,\dots, a_{n+3})_{i_1,\dots, i_n} := A(a_{i_n}) \circ A(a_{i_n +1}) \backslash A(a_{i_n} \circ a_{i_n +1})
\end{displaymath}
where $A(a) := (a_1, \dots, a_{i_{n}-1},a,a_{i_n + 2},\dots, a_{n+3})_{i_1,\dots, i_{n-1}}$ and $1 \leq i_n \leq n+2$. 

A bracket of weight $n$ is an expression in $n$ indeterminates formed by repeatedly applying commutators, associators and deviations, and in which every indeterminate appears only once. A filtration of a loop $Q$ by normal subloops $Q = Q_1 \supseteq Q_2 \supseteq \cdots $ is said to be an N-sequence if for any bracket  $P(a_1,\dots, a_n)$ of weight $n$  and any $i_1,\dots, i_n$ we have $P(Q_{i_1},\dots, Q_{i_n}) \subseteq Q_{i_1+\cdots + i_n}$. Since $[Q_i,Q_i], (Q_i,Q_i,Q_i) \subseteq Q_{i+1}$, the quotients $\s_i:=Q_i/Q_{i+1}$ are abelian groups.  Thus, any bracket of weight $n$ induces a homomorphism of abelian groups $\s_{i_1} \otimes_{\Z} \cdots \otimes_{\Z} \s_{i_n}\to \s_{i_1+\cdots + i_n}$ and it defines an $n$-ary operation $p\colon \s \otimes_\Z \cdots \otimes_\Z \s \to \s$ where $\s := \oplus_{i\geq 1} \s_i$. 

Natural N-sequences for a loop are the commutator-associator and the dimension filtrations \cites{Mo06,MoPe07} among others. Starting with $\gamma_1 Q := Q$, the $n$-th term $\gamma_n Q$ of the commutator-associator filtration of $Q$ is the minimal normal subloop containing $[\gamma_i Q, \gamma_j Q]$ ($i+j \geq n$), $(\gamma_i Q, \gamma_j Q, \gamma_k Q)$ ($i+j+k \geq n$) and $(\gamma_{p_1}Q,\dots, \gamma_{p_{l+3}}Q)_{i_1,\dots,i_l}$  ($p_1 + \cdots + p_{l+3} \geq n$). On the graded vector space   $\s:=\bigoplus_{i\geq 1}\field \otimes_{\Z}\gamma_i Q/\gamma_{i+1}Q$, the operations $p_{n,m}$ defined by the deviations
\begin{displaymath}
	P_{n,m}(x_1,\dots,x_n; y_1, \dots, y_m;z) = (x_1,\dots, x_n, y_1, \dots, y_m,z)_{\underbrace{1,\dots, 1}_{n-1},\underbrace{n+1,\dots,n+1}_{m-1}}
\end{displaymath}
lead to $\field$-multilinear operations
\begin{equation}\label{eq:J_brackets}
	\langle x_1,\dots, x_n; y,z \rangle := p_{n,1}(x_1,\dots, x_n; z;y) - p_{n,1}(x_1,\dots, x_n; y;z).
\end{equation}
In case that the characteristic of $\field$ is zero then $\s$  endowed with these operations,
and with a multioperator $\Phi$ defined from $p_{n,m}$, is  a Sabinin algebra \cite{Mo08}. Other N-sequences might lead to Sabinin algebras in a similar way. In this paper we will compute the Sabinin brackets associated to the natural filtration $\J_1(R) \supseteq \J_2(R) \supset \cdots$ where $\J_n(R):=\{ t + \sum_{i\geq n}\alpha_i t^{i+1} \mid \alpha_i \in R\}$.

In \cite{ShU02} Shestakov and Umirbaev brought Sabinin algebras to an algebraic ground. In the same way as Lie algebras appear inside associative algebras with the commutator product $[x,y]:= xy - yx$  (Poincar\'e-Birkhoff-Witt Theorem), Shestakov and Umirbaev defined Sabinin brackets and a multioperator out of the product of any nonassociative algebra, giving rise to a functor $\SU$ from the category of nonassociative algebras to the category of Sabinin algebras. Our description of the Sabinin brackets associated to $\J(R)$ relays on an auxiliary nonassociative algebra and the Shestakov-Umirbaev  functor $\SU$.

In this initial approach to the loop of formal power series under composition, it will become apparent that a necessary and sufficient condition for (\ref{eq:Nottingham}) to define a group is $R[R,R] = 0$. In this case we get a natural structure of Lie algebra for $R[[t]]$ determined by
\begin{equation}\label{eq:Wronskian_product}
\langle f(t),g(t) \rangle = g(t)'f(t) - f(t)'g(t).
\end{equation}	
where $f':= \frac{d}{dt}(f)$. If $R$ is the ground field $\field$ then we recover the Lie algebra of vector fields on the affine line. In general, given any unital commutative algebra $\phi$ other than $R[[t]]$ and a derivation $\partial$ of $\phi$, (\ref{eq:Wronskian_product}) also defines a (differential) Lie algebra structure on $\phi$. These Lie algebras are called Lie algebras of vector fields on a line \cite{Po17} and they are precisely those Lie algebras that  embed into their Wronskian envelopes (see \cite{Po17} for details), i.e. they are Wronskian special. 

The standard identity $\St_{n+1}$ is defined by the alternating sum
\begin{displaymath}
	\St_{n+1}(x_1,\dots,x_n,z):= \sum_{\sigma} (-1)^{\sigma} [x_{\sigma(1)},[\cdots,[x_{\sigma(n)},z]]]
\end{displaymath} 
where $\sigma$ runs the symmetric group of degree $n$. It is well-known \cite{Be79} that any Lie algebra of vector fields on a line satisfies $\St_5$, i.e. it is a $\St_5$-algebra, so not every Lie algebra is Wronskian special. In fact, whether  $\St_5$-algebras are the same as Wronskian special algebras, even in the case of characteristic zero, is a long-standing problem already considered by Kirillov, Ovsienko and Udalova in \cite{KOU84}. Razmyslov gave an affirmative answer in \cite{Ra85} for simple Lie algebras and later, in \cite{PoRa16}, in a joint work with Pogudin they extended this affirmative answer to prime Lie algebras.

We will show that if $R$ is not commutative but $R[R,R] = 0$ then $(R[[t]],\langle - , - \rangle)$ satisfies $\St_6$, although it might not be a $\St_5$-algebra in general. To ensure that $(R[[t]],\langle - , - \rangle)$  is a $\St_5$-algebra, $[R,R]R^3 = 0$ must be required in addition. It would be interesting to check whether these algebras are Wronskian special or not. 

We will conclude this paper with and appendix on one-sided loops natural in this context.

A word about notation and conventions. The characteristic of the base field $\field$ is zero. $R[[t]]$ (resp. $R[t]$)  denotes the algebra of formal power series (resp. polynomials) in the indeterminate $t$ with coefficients in the algebra $R$.  Given $I=(i_1,\dots,i_n) \in \N^n$ and $\alpha_0,\alpha_1,\dots$ in an associative algebra,  we will use the notation $l(I):= n$, $\vert I \vert := i_1 + \cdots + i_n$ and $\alpha_I:= \alpha_{i_1}\cdots \alpha_{i_n}$.  The left and right multiplication operators by $x$, with respect to a product $xy$, will be denoted by $L_x$ and $R_x$. Sometimes,  instead of juxtaposition $\circ, *, \dots$ will be preferable for some  products, and we will use $L^{\circ}_x, R^{\circ}_x, L^*_x, R^*_x, \dots$ accordingly.

%
\section{The loop $(\Jg(\ring), \circ)$ of formal power series under composition}

Let $\ring := \bigoplus_{n \geq 1} \ring_n$ be a graded associative algebra over a field $\field$ and $\hat{\ring}$ its  completion with respect to the grading. Our object of study is the set
\begin{displaymath}
\Jg(\ring):= 1 + \hat{\ring},
\end{displaymath}
where $1$ is a formal unit element, with the product modeled on the composition of formal power series\footnote{We identify the formal power series $t + \alpha_1 t^2 + \alpha_2 t^3 + \cdots$ with $1 + \alpha_1 + \alpha_2 + \cdots$.}
\begin{equation}\label{eq:substitution}
\left(\sum_{m\geq 0} \alpha_m\right) \circ \left(\sum_{n \geq 0} \beta_n\right):= \sum_{m \geq 0} \alpha_m \left(\sum_{n \geq 0} \beta_n \right)^{m+1} = \sum_{k \geq 0}  \gamma_k 
\end{equation}
where $\alpha_0 = 1 = \beta_0$ and 
\begin{equation}\label{eq:gamma}
\gamma_k := \sum_{m\geq 0} \alpha_m  \left(\sum_{\substack{l(J) = m+1 \\ \vert J \vert + m = k}} \beta_J\right).
\end{equation}

\begin{proposition}\label{prop:loop}
	$(\Jg(\ring), \circ )$ is a loop with unit element $1$.
\end{proposition}
\begin{proof}
	On the one hand, if $\alpha_m := 0$ for all $m \geq 1$ in (\ref{eq:gamma}) then $\gamma_k = \alpha_0 \beta_k = \beta_k$, which implies that $1$ is a left unit element. If $\beta_m = 0$ for all $m \geq 1$ then $(\ref{eq:gamma})$ gives $\gamma_k = \alpha_k \beta_0 = \alpha_k$, so $1$ is also a right unit element. On the other hand, the coefficient of $\alpha_k$ (resp. $\beta_k$) in (\ref{eq:gamma}) is $1$. Thus, given $\sum_{n\geq 0} \beta_n$ (resp. $\sum_{m\geq 0} \alpha_m$) and $\sum_{k\geq 0} \gamma_k$  there exists a unique solution $\sum_{m\geq 0} \alpha_m$ (resp. $\sum_{n\geq 0} \beta_n$) of (\ref{eq:substitution}). This proves that the left and right multiplication operators by elements of $\Jg(\ring)$ are bijective.
\end{proof}

\begin{example}
In general, if $\ring$ is not commutative then $(\Jg(\ring),  \circ)$  is not a group. For instance,  the component of degree $3$  of  $((1 + \alpha_1) \circ (1 + \beta_1)) \circ (1 + \gamma_1)$ is 
\begin{displaymath}
 \alpha_1 \gamma^2_1 + \alpha_1 \beta^2_1 + \beta_1 \gamma^2_1 + 6 \alpha_1 \beta_1 \gamma_1
\end{displaymath}
while the corresponding component of $(1 + \alpha_1) \circ ((1 + \beta_1) \circ (1 + \gamma_1))$ is
\begin{displaymath}
  \alpha_1 \gamma^2_1 + \alpha_1 \beta^2_1 + \beta_1 \gamma^2_1 +  5 \alpha_1 \beta_1 \gamma_1 + \alpha_1 \gamma_1 \beta_1,
\end{displaymath}
i.e. they differ by  $\alpha_1(\beta_1 \gamma_1 - \gamma_1 \beta_1)$, which is nonzero in general. \hfill$\square$
\end{example}

%
\subsection{The role of the coefficients.}\label{subsec:coefficients} 

To understand the role of the commutativity of the coefficients when composing formal power series, consider $R$ to be a unital associative commutative   algebra and $R[[t]]$  the algebra of formal power series in the indeterminate $t$. Given 
$s := s(t):= \sum_{i \geq 0} s_i t^i \in R[[t]]$ with $s_0 := 0$ and invertible $s_1$, the substitution
\begin{align*}
\sigma_s \colon R[[t]] & \rightarrow R[[t]] \\
a(t) & \mapsto a(t)^{\sigma_{s}}:=a(s(t))
\end{align*}
is an automorphism of $R[[t]]$, which implies $\sigma_a \sigma_b = \sigma_{a(b(t))}$ --notice that substitutions are forced to act on the right to get this formula. Thus, in this case the associativity of formal power series under substitution mirrors the associativity of the composition of automorphisms of $R[[t]]$. However, if $R$ is noncommutative then the substitution $\sigma_{s}$ is no longer an automorphism. For instance,
\begin{displaymath}
(t (\alpha_1 t))^{\sigma_{s}} = (\alpha_1 t^2)^{\sigma_{s}} = \alpha_1 s^2 \quad  \text{but} \quad t^{\sigma_{s}} (\alpha_1 t)^{\sigma_{s}} = s \alpha_1 s.
\end{displaymath}
Therefore, while substitution still defines a binary operation, it might not be associative anymore. Consider the group $G$ generated by all the elements $\sigma_s$ with $s = \sum_{i \geq 0} s_i t^i$, $s_0 = 0$ and invertible $s_1$, and let $H$ be the stabilizer of $t$ in $G$. The set $Q$ of all $\sigma_s$ in $G$ is a right transversal of $H$, i.e. any $\sigma \in G$ can be uniquely decomposed as $\sigma = \sigma_0 \sigma_s$ with $\sigma_0 \in H$ and $\sigma_s \in Q$. Clearly $H\sigma_a \sigma_b = H \sigma_{a(b(t))}$. Thus, $a(b(t))$ mirrors a corresponding product on the transversal $Q$: given $\sigma_a,\sigma_b \in Q$ consider the product of $\sigma_a$ and $\sigma_b$ to be the unique element in $Q \cap H \sigma_a\sigma_b$. Since we cannot ensure that this element, which is $\sigma_{a(b(t))}$, is the composition $\sigma_a\sigma_b$, associativity might be lost.

\begin{example}
	Even if $\ring$ is noncommutative, $\Jg(\ring)$ might be a group. For instance, consider the algebra of  $3 \times 3$ upper triangular matrices with the gradation $\ring_1 := \field E_{12} + \field E_{23}$ and $\ring_2 := \field E_{13}$, where $E_{ij}$ stands for the matrix whose only nonzero element is a $1$ placed in the position $(i,j)$. It is easy to check that 
	\begin{displaymath}
	 (1 + a) \circ (1 + b) = 1 + a + b + 2ab,
	\end{displaymath}
	which defines a group structure. However, if $\ring$ is the algebra of $n \times n$ upper triangular matrices then $\Jg(\ring)$ is no longer a group in general. \hfill$\square$
\end{example}

\begin{proposition}
	Let $\field$ be a field of characteristic zero and $\ring = \oplus_{n \geq 1} \ring_n$ a graded associative  $\field$-algebra. Then $(\Jg(\ring),\circ)$ is a group if and only if $\ring[\ring,\ring] = 0$.
\end{proposition}
\begin{proof}
	Let $\lambda, \lambda', \lambda''$ be scalars in $\field$, which is infinite, and $\alpha_i \in \ring_i, \beta_j \in \ring_j, \gamma_k \in \ring_k$. We will compute the coefficient of $\lambda\lambda'\lambda''$ in $s := ((1 + \lambda\alpha_i) \circ (1 + \lambda'\beta_j)) \circ (1 + \lambda''\gamma_k) - (1 + \lambda\alpha_i) \circ ((1 + \lambda'\beta_j) \circ (1 + \lambda''\gamma_k))$. We use the symbol $\equiv$ to indicate that two elements share the same coefficient in  $\lambda\lambda'\lambda''$. We have
	\begin{align*}
		s & \equiv (1 + \lambda' \beta_j + \lambda \alpha_i (1 + \lambda'\beta_j)^{i+1})\circ ( 1 + \lambda'' \gamma_k) \\
		& \quad - (1+ \lambda \alpha_i) \circ (1 + \lambda'' \gamma_k + \lambda' \beta_j (1+ \lambda''\gamma_k)^{j+1})\\
		& \equiv (i+1) \lambda \lambda' \alpha_i   \beta_j (1 + \lambda'' \gamma_k)^{i+j+1} \\
		& \quad - \lambda \alpha_i (1 + \lambda'' \gamma_k + \lambda'\beta_j + \lambda'\lambda''(j+1) \beta_j\gamma_k)^{i+1}\\
		& \equiv (i+1)(i+j+1)\lambda\lambda'\lambda'' \alpha_i \beta_j \gamma_k \\
		& \quad - \lambda\alpha_i \left(\binom{i+1}{2}\lambda'\lambda''\beta_j\gamma_k + \binom{i+1}{2}\lambda'\lambda''\gamma_k\beta_j + (i+1)(j+1)\lambda'\lambda''\beta_j\gamma_k \right)\\
		& \equiv \binom{i+1}{2}\lambda\lambda'\lambda''\alpha_i[\beta_j,\gamma_k].
	\end{align*}
	Therefore if $\Jg(\ring)$ is a group then $\ring[\ring,\ring] = 0$. Conversely, given $1+a,1+b,1+c \in \Jg(\ring)$, consider $s' := ((1+a)\circ (1+b))\circ (1+c) - (1+a) \circ ((1+b)\circ (1+c))$. We have
	\begin{align*}
		s' &= (1+b + a\circ (1+b))\circ (1+c) - (1+a) \circ (1+c + b \circ (1+c)) \\
		&= (1+c + b \circ (1+c) +(a \circ (1+b))\circ (1+c))\\
		 &\quad - (1+ c + b \circ(1+c) +a \circ (1+c+b \circ(1+c))) \\
		&= (a \circ (1+b))\circ (1+c) - a \circ (1+c+b \circ(1+c)).
	\end{align*}
	Thus, without loss of generality we can assume $a = \alpha_i \in \ring_i$. We have
	\begin{align*}
		s' &=  (\alpha_i (1+b)^{i+1})\circ (1+c) - \alpha_i (1+c+b \circ(1+c))^{i+1}\\
		&= \alpha_i \left( (1+b)^{i+1}\circ (1+c) - (1+c+b \circ(1+c))^{i+1}\right).
	\end{align*}
	Since $\ring[\ring,\ring] = 0$, the presence of $\alpha_i$ allows us to assume that the components of $b$ and $c$ not only associate but they also commute, and in that case it is easy to check that $(1+b)^{i+1}\circ (1+c) - (1+c+b \circ(1+c))^{i+1} = 0$. 
\end{proof}

%
\subsection{The loop $(\Jg(\ring),*)$.} 

We introduce a new non-associative continuous product on $\hat{\ring}$ by
\begin{equation}\label{eq:*}
 \alpha_m  * \beta_n =  (m+1) \alpha_m \beta_n
\end{equation}
for homogeneous $\alpha_m, \beta_n$, and we assume $1$ to be a formal unit element for this new product. It would be misleading to think of $1$ as an element of degree $0$ since in that case (\ref{eq:*}) would imply that $1$ is a left unit element but not a right one. 

\begin{proposition}
	$(\Jg(\ring), * )$ is a loop with unit element $1$.
\end{proposition}
\begin{proof}
The product of $\sum_{m\geq 0} \alpha_m$ and $\sum_{n \geq 0} \beta_n$ is given by
\begin{displaymath}
  1 + \sum_{m\geq 1} \alpha_m + \sum_{n \geq 1} \beta_n + \sum_{n,m \geq 1} (m+1) \alpha_m \beta_n
\end{displaymath}	
and the coefficient of degree $k \geq 1$ is 
\begin{displaymath}
\alpha_k + \beta_k + \sum_{\substack{m + n = k \\ n,m\geq 1}} (m+1)\alpha_m \beta_n.
\end{displaymath}
As in the proof of Proposition~\ref{prop:loop} this easily leads to the bijectivity of the left and right multiplication operators by elements of $\Jg(\ring)$.
\end{proof}

Notice that $(\Jg(\ring), *)$ is isomorphic to $\hat{\ring}$ with the operation $(\alpha,\beta) \mapsto \alpha + \beta + \alpha * \beta$, which is the natural way of obtaining a local loop around $0$ out of any (not necessarily unital) algebra. Next we will compare the loops $(\Jg(\ring),\circ)$ and $(\Jg(\ring), *)$.

\begin{proposition}\label{prop:differential}
	For any  $a \in \Jg(\ring)$ and $b \in \hat{\ring}$  we have
	\begin{displaymath}
	\left.\frac{d}{dt}\right\vert_{t=0}L^\circ_a(1 + t b) = \left.\frac{d}{dt}\right\vert_{t=0}L^*_a(1 + t b).
	\end{displaymath}
\end{proposition}
\begin{proof}
	Consider $a = \sum_{m \geq 0} \alpha_m$ with $\alpha_0 = 1$. Clearly
	\begin{align*}
		\left.\frac{d}{dt}\right\vert_{t=0}L^\circ_a(1 + t b) &= \left.\frac{d}{dt}\right\vert_{t=0} \sum_{m \geq 0} \alpha_m (1+tb)^{m+1} \\
		&= \sum_{m \geq 0} (m+1) \alpha_m  b = a* b = \left.\frac{d}{dt}\right\vert_{t=0}L^*_a(1 + t b).
	\end{align*}
\end{proof}

In the approach by Mikheev and Sabinin, any local analytic loop induces a parallel transportation $\tau^x_y(\xi_x) := dL_y\vert_e (dL_x\vert_e)^{-1}(\xi_x)$ so that the associated right monoalternative loop is obtained by $x \times y := \exp_x \tau^e_x \exp^{-1}_e(y)$. Two local analytic loops $x \circ y$ and $x * y$ for which the differential of $L^{\circ}_x$ and $L^{*}_x$ at $e$ agree define the same parallel transportation and also the same associated right monoalternative loop. Thus, the Sabinin brackets for both loops are the same. Proposition~\ref{prop:differential} suggests that we can compute the Sabinin brackets for $(\Jg(S),\circ)$ by means of the simpler product $x*y$. While this is the leading idea behind our computations, since $\J^g(S)$ has no structure of analytic manifold, we are concerned with the Sabinin brackets related to the natural filtration of $\Jg(\ring)$ rather than with the Sabinin brackets that appear in the geometrical context where the theory originally arose.

Notice that if we set $a \bullet b:= b + a \circ (1+b)$ for any $a,b \in \hat{\ring}$ then we obtain a loop structure on $\hat{\ring}$ isomorphic to $(\Jg(\ring), \circ)$. The connection between the loop $(\Jg(\ring),\circ)$ and the simpler structure $(\hat{\ring},*)$ is due to the linearity  in $a$ of $a \bullet b - b$.
\begin{proposition}
	Let $(V,a \bullet b)$ be an analytic loop structure on a finite-dimensional real vector space $V$ such that $a \bullet b - b$ is linear in $a$. Then 
	\begin{displaymath}
	 a* b : = dL^{\bullet}_a \vert_0 (b) - b
	\end{displaymath}
	is a bilinear product, $ab := a + b + a*b$ defines a local analytic loop at $0$ and 
	\begin{displaymath}
		dL^{\bullet}_a\vert_0 = dL_a\vert_0.
	\end{displaymath}
\end{proposition}
\begin{proof}
	The linearity of $a \bullet b - b$ in $a$ implies $0 \bullet b - b = 0$, i.e. $0 \bullet b = b$, thus $0$ is he unit element of $(V, a \bullet b)$. We have
	\begin{displaymath}
		a* b = dL^{\bullet}_a\vert_0(b) - b= \left.\frac{d}{d \lambda }\right\vert_{\lambda =0} (a \bullet (\lambda b) - (\lambda b))
	\end{displaymath}
	which is linear in $b$ and, by hypothesis, also in $a$. Clearly
	\begin{displaymath}
		dL_a\vert_0 (b) =  \left.\frac{d}{d \lambda }\right\vert_{\lambda =0} a (\lambda b) =  \left.\frac{d}{d \lambda }\right\vert_{\lambda =0} (a + \lambda b + a*(\lambda b))= b + a *b = dL^{\bullet}_a\vert_0 (b).
	\end{displaymath}
\end{proof}

%
%
%
\section{Shestakov-Umirbaev brackets for the product $*$} 
 
In \cite{MoPe10} it was proved that the Sabinin brackets of the loop  around $0$ with product $x + y + x*y$ associated to any nonassociative algebra with product $x*y$ agree with those obtained through the Shestakov-Umirbaev functor $\SU$. By Proposition~\ref{prop:differential} we only have to compute these brackets for the product (\ref{eq:*}) to obtain the Sabinin brackets for analytic $(\Jg(S),\circ)$. So, let us first recall how the functor functor $\SU$ was defined in \cite{ShU02}.

Let $(A,*)$ be an algebra over $\field$. Given $I:=(i_1,\dots,i_m)$ and $a_{i_1}, \dots, a_{i_m} \in A$, let us define
\begin{equation}\label{eq:tensors}
	\abf_I:= a_{i_1} \otimes \cdots \otimes a_{i_m}, \quad a^*_I:= ((a_{i_1} * a_{i_2})\cdots ) * a_{i_m} 
\end{equation}
and 
\begin{displaymath}
	\sum \abf_{I(1)} \otimes \abf_{I(2)} := \sum_{p = 0}^m\sum_{\sigma} \big(a_{i_{\sigma(1)}} \otimes \cdots \otimes a_{i_{\sigma(r)}}\big) \otimes \big(a_{i_{\sigma(r+1)}} \otimes \cdots \otimes a_{i_{\sigma(m)}}\big)
\end{displaymath}
where the sum runs on all $(r,m-r)$ shuffles $\sigma$. This formula corresponds to the comultiplication $\Delta$ on the tensor algebra $T(A)$ when we impose the elements $a\in A$ to be primitive (i.e. $\Delta(a) = a \otimes 1_{\field} + 1_{\field} \otimes a$) and $\Delta \colon T(A) \to T(A) \otimes T(A)$ to be a homomorphism of algebras. Expressions such as $\sum a^*_{I(1)} \otimes \abf_{I(2)}$ have the obvious meaning. The coassociativity of $\Delta$ justifies the notation
\begin{displaymath}
	\sum \abf_{I(1)} \otimes \cdots  \otimes \abf_{I(k)}
\end{displaymath}
when we apply $\Delta$ to $\abf_I$ $k-1$ times.

Shestakov-Umirbaev $p$-operations  on $(A,*)$ are defined recursively by the following fundamental formula
\begin{displaymath}
	(a^*_{I} * a^*_J) * a - a^*_I * (a^*_J * a) = \sum (a^*_{I(1)} * a^*_{J(1)}) * p(\abf_{I(2)}; \abf_{J(2)}; a)
\end{displaymath}
These operations induce a Sabinin algebra structure on $A$ by 
\begin{align*}
	\langle a,b \rangle &:= b*a- a*b\\
	\langle a_1,\dots, a_m; b,c  \rangle &: = p(a_1,\dots,a_m;c;b) - p(a_1,\dots, a_m;b;c) \\
	\Phi(a_1,\dots,a_m; b_1,\dots, b_{n+1}) &:=\\
	&\hskip -2cm  \frac{1}{m!(n+1)!} \sum_{\sigma \in S_{m}, \tau \in S_{n+1}} p(a_{\sigma(1)},\dots, a_{\sigma(m)}; b_{\tau(1)},\dots;b_{\tau(n+1)}).
\end{align*}
Observe that $\langle 1; a,b \rangle = 0$. To unify notation it is customary to set $\langle 1; a,b \rangle$ as $\langle a, b \rangle$. However, here \emph{we shall not follow that convention in Proposition~\ref{prop:brackets}}.

Let us specialize this construction to our context. Let  $A = \bigoplus_{n \in \Z} A_n$ be a graded unital associative algebra over $\field$ and define a new product on $A$ by 
 \begin{equation}\label{eq:prod_*}
 a_m * a_n := (m+1)a_m a_n.
 \end{equation}
 Given $I:=(i_1,\dots,i_m)$ and $a_{i_1}, \dots, a_{i_m} \in A$ ($a_{i_j} \in A_{i_j}$), in addition to (\ref{eq:tensors}),  let us define $a_I := a_{i_1}\cdots a_{i_m}$.
 
\begin{proposition} \label{prop:brackets}	 
Given $I:=(i_1,\dots, i_m)$ and $a_{i_1}, \dots, a_{i_m}, b,c \in A$ with $a_{i_j} \in A_{i_j}$ we have
	\begin{displaymath}
		\langle \abf_I; b,c \rangle = \sum_{k, l(I(1)),\dots, l(I(k)) \geq 1} (-1)^{k+1} \vert I(k) \vert N(I(1)) \cdots N(I(k))  a_{I(1)} \cdots a_{I(k)} [c,b]
	\end{displaymath}
	where $[c,b]:=cb-bc$, 
	 \begin{displaymath}
	 N(I):= (i_1 +1)(i_1+i_2+1)\cdots (i_1+\cdots + i_m+1)
	 \end{displaymath}
	 if $m \geq 1$ and  $N(I) := 1$ if $m = 0$.
\end{proposition}
\begin{proof}
	By  definition, $(a^*_I * c) * b - (a^*_I * b) * c = a^*_I* \langle b, c \rangle + \sum a^*_{I(1)}* \langle \abf_{I(2)}; b,c \rangle$. Hence
	\begin{align*}
		N(I) (i_1 &+ \cdots + i_m + \vert c \vert +1 ) a_I cb - 	N(I) (i_1+ \cdots + i_m + \vert b \vert +1 ) a_I bc \\
		&= N(I)(\vert c \vert + 1) a_Icb - N(I)(\vert b \vert +1) a_Ibc + \sum N(I(1)) a_{I(1)} \langle \abf_{I(2)}; b,c \rangle
	\end{align*}
	and we get
	\begin{displaymath}
		\vert I \vert N(I)  a_I[c,b] = \sum N(I(1)) a_{I(1)} \langle \abf_{I(2)}; b,c \rangle.
	\end{displaymath}
	Solving for $\langle \abf_I; b,c \rangle$ we obtain
	\begin{displaymath}
		 	\langle \abf_I; b,c \rangle = \sum_{k, l(I(1)),\dots, l(I(k)) \geq 1} (-1)^{k+1} \vert I(k) \vert N(I(1)) \cdots N(I(k))  a_{I(1)} \cdots a_{I(k)} [c,b].
	\end{displaymath}
\end{proof}

In Section~\ref{sec:N-sequence} we will use Proposition~\ref{prop:brackets} to identify the Sabinin brackets related to the natural filtration of $\Jg(\ring)$. Some brackets on $(A,*)$ are
	\begin{align*}
		\langle a, b \rangle &= (\vert b \vert + 1) ba - (\vert a \vert + 1) ab,\\
		\langle a_{i}; b,c \rangle &= i(i+1) a_i[c,b] \quad \text{and}\\
		\langle a_i, a_j ; b,c \rangle &= i(i+1)(i+2j+1) a_ia_j[c,b] - i(i+1)(j+1) a_ja_i[c,b].
	\end{align*}

\begin{example}
	Let $A:= \field[t,t^{-1}]$ be the algebra of Laurent polynomials with the natural $\Z$-grading. The only nonzero bracket $\langle -; -,- \rangle$ is the binary one, which is determined by
	\begin{displaymath}
	\langle t^i, t^j \rangle = (j-i) t^{i+j},
	\end{displaymath}
	the product of the Witt Lie algebra. \hfill$\square$
\end{example}

%
\subsection{Examples of nonnilpotent Sabinin algebras with trivial $n$-ary brackets.}\label{subsec:trivial_brackets}

 It is well known that the tangent algebras of local Lie groups are Lie algebras, i.e. Sabinin algebras for which all multilinear operations other than the binary product $\langle -, - \rangle$ vanish. With the help of Proposition~\ref{prop:brackets}, given $n \geq 1$, it is easy to construct Sabinin algebras that are not nilpotent but for which the brackets $\langle -;-,- \rangle$ of ariety $m$ vanish just if $m \geq n$ (or $m=0$ since $\langle a_1,\dots,a_m ; b,c \rangle = 0$ if $m=0$ unless we identify these brackets with $\langle b,c \rangle$). Loosely speaking, these examples are $n-1$ steps far from being Lie algebras. 

We first consider the associative $\field$-algebra $\field[e]$ generated by an element $e$ subject to the relation $e^{n+1} = e^n$. We also consider a $\field[e]$-bimodule $M$ with a basis $\{v^0,\dots, v^{n-1}\}$ and actions given by
\begin{displaymath}
ev^i := v^{i+1}\quad \text{and}\quad v^i e = v^i \quad (i=0,1,\dots, n-1)
\end{displaymath}
where $v^n:= 0$. Let $R:= \field[e] \oplus M$ be the split-null extension of $\field[e]$ by $M$. Thus $(\alpha + u)(\beta + v) = \alpha\beta + \alpha v + u \beta $ for any $\alpha, \beta \in \field[e]$ and $u,v \in M$. Since
\begin{displaymath}
[e,v^i] = v^{i+1}-v^i
\end{displaymath}
then $[R,R] = M$. Thus
\begin{equation}\label{eq:null-n}
R^n[R,R] = 0.
\end{equation}
Let $\ring_i:= R$ be a copy of $R$ for $i=-n,\dots, 0$ and $\ring:= \oplus_{i=-n}^0 \ring_i$. To distinguish elements, given $a,b,\dots \in R$, we use subindexes $a_i, b_j,\dots$ to indicate the component $\ring_i, \ring_j,\dots$ to which they belong. We will also assume that $\ring_k = 0$ if $k \not\in \{ -n,\dots, 0\}$. There are two natural products on $\ring$. On the one hand, we can define the product $a_i \beta_j := (\alpha\beta)_{i+j}$ on $\ring$ to obtain an associative $\Z$-graded algebra; on the other hand, we can consider the product $\alpha_i * \beta_j := (i+1)(\alpha\beta)_{i+j}$. By Proposition~\ref{prop:brackets} and (\ref{eq:null-n}) we have
\begin{displaymath}
\langle \underbrace{\ring,\ring,\dots, \ring}_m; \ring, \ring \rangle =0 \quad \text{if } m \geq n. 
\end{displaymath}
However, if $1 \leq m < n$ then
\begin{equation}\label{eq:nonzero_bracket}
\langle \underbrace{e_{-2},e_{-1},\dots, e_{-1}}_m; \ring_0,\ring_0 \rangle \neq 0
\end{equation}
To compute this expression we look at the sum in Proposition~\ref{prop:brackets} with $I=(-2,-1,\dots, -1)$ in accordance to (\ref{eq:nonzero_bracket}). If $k \geq 2$ then in the corresponding summand appears at least one $I(i)$ of the form $I(i)=(-1,\dots,-1)$. Since $N(I(i)) = 0$ then that summand vanishes. Therefore, 
\begin{align*}
\langle e_{-2},e_{-1},\dots, e_{-1}; \ring_0,\ring_0 \rangle &= \vert I \vert N(I) e^m_{-m-1}[\ring_0,\ring_0] \\
&= (-1)^{m+1} (m+1)! (e^m M)_{-m-1} \neq 0.
\end{align*}
Moreover, 
\begin{displaymath}
\langle \ring_0,\langle \ring_0,\dots,\langle \ring_0,\ring_0 \rangle\rangle \rangle = [R,[R,\dots,[R,R]]] = M
\end{displaymath}
implies that the Sabinin algebra is not nilpotent. 

%
\section{The N-sequence $\Jg_1(\ring) \supseteq \Jg_2(\ring) \supseteq \cdots$  and its Sabinin brackets}
\label{sec:N-sequence}

%
\subsection{The N-sequence $\Jg_1(\ring) \supseteq \Jg_2(\ring) \supseteq \cdots$.}

The loop $(\Jg(\ring),\circ)$ has a natural filtration $\Jg_1(\ring) \supseteq \Jg_2(\ring) \supseteq \cdots$ where
\begin{equation}\label{eq:J_i}
\Jg_i(\ring):=\left\{ 1 +  \sum_{k \geq i } \alpha_k \mid \alpha_k \in \ring_k \right\}  \quad (i \geq 1).
\end{equation}
Clearly $\Jg_i(\ring)$ is the kernel of the natural projection $\Jg(\ring) \rightarrow \Jg(\ring/I)$ where $I$ is the ideal $ I:= \ring_{i} \oplus \ring_{i+1} \oplus \cdots$ of $\ring$, therefore it is a normal subloop.  Modulo $\Jg_{i+1}(\ring)$ we have $$(1 + \alpha_i) \circ (1 + \beta_i) \equiv 1 + \alpha_i + \beta_i,$$ so  $(\Jg_i(\ring)/\Jg_{i+1}(\ring), \circ)$ is an abelian group isomorphic to $(\ring_i,+)$. Even more, $$\xi (1 + \alpha_i )\Jg_{i+1}(\ring) := (1 + \xi \alpha_i )\Jg_{i+1}(\ring)$$ defines a $\field$-vector space structure on $\s_i:= \Jg_i(\ring)/\Jg_{i+1}(\ring)$ isomorphic to $\ring_i$. We will prove that this filtration is an N-sequence and that $\s:= \oplus_{i\geq 1} \s_i$ is a Sabinin algebra with the brackets defined in (\ref{eq:J_brackets}). Moreover, we will show that with the identification 
\begin{equation}\label{eq:identification}
\s = \bigoplus_{i\geq 1} \Jg_i(\ring)/\Jg_{i+1}(\ring) \cong \bigoplus_{i\geq 1} \ring_i = \ring
\end{equation}
these brackets are given by the formula in Proposition~\ref{prop:brackets}. To this end, there is no loss of generality in assuming that $\ring$ is the (nonunital) associative algebra $\coef$ freely generated by 
$$\Omega := \{ \alpha_{i,k} \mid i, k \geq 1\} \cup \{ \alpha_k, \beta_k, \gamma_k, \dots \mid k \geq 1\}.$$ 
$\coef$ is a graded algebra with the gradation determined by 
$$\vert \alpha_{i,k}\vert := \vert \alpha_k \vert := \vert \beta_k\vert := \cdots := k.$$ 
We will identify the vector spaces $\Jg_i/\Jg_{i+1}$ and $\coef_i$, where $\Jg_i$ stands for $\Jg_i(\coef)$. We will consider consider $\coef^\sharp:= \field 1 \oplus \coef$ so that $(\coef^\sharp, ab)$ is the unital closure of $(\coef,ab)$ and $(\coef^\sharp, a*b)$ is the unital closure of $(\coef,a*b)$, where $*$ is the product defined in (\ref{eq:prod_*}). Be aware that $\alpha_i * \beta_j = (i+1)\alpha_i\beta_j$ but $\alpha_i*1 = \alpha_i$.

Given  $w = 1 + \sum_{i \geq 1} w_i$ with $w_i \in \coef_i$,  the \emph{depth} of $w$ is 
$$d(w) := \min \{ i \geq 1 \mid w_i \neq 0 \}$$
for  $w \neq 1$ and $d(1) := \infty$. Given $\alpha \in \Omega$ and $w \in \Jg(\coef)$, we can evaluate $w$ at $\alpha = 0$ to obtain a new element $w \vert_{\alpha \rightarrow 0} \in \Jg(\coef)$. The \emph{support} of $w \in \Jg(\coef)$ is $$\supp(w):= \{ \alpha \in \Omega \mid w \neq w\vert_{\alpha \rightarrow 0}\}.$$
Let $\fl$ be the loop freely generated by $x,y,z$ and $ x_i,y_i,z_i$ ($i=1,2,\dots$) with unit element $e$. We say that $w(x_1,\dots,x_m) \in \fl$ is \emph{balanced} with respect to $x_i$ if $w(x_1,\dots,x_m)\vert_{x_i \rightarrow e}:= w(x_1,\dots, x_{i-1},e,x_{i+1},\dots,x_m) = e$.  We say that $w(x_1,\dots,x_m)$ is \emph{balanced} if it is balanced with respect to $x_i$ for all $i=1,\dots,m$.
\begin{proposition}\label{prop:offset}
	Let $w(x_1,\dots,x_m) \in \fl$ be balanced and $a_i = 1 + \sum_{j \geq d(a_i)} \alpha_{i, j} \in \Jg(\coef)$, ($i = 1, \dots, m$). Then 
	\begin{displaymath}
	d(w(a_1,\dots, a_m)) \geq d(a_1) + \cdots + d(a_m).
	\end{displaymath}
	Moreover, the component of degree $d(a_1) + \cdots + d(a_m)$ of $w(a_1,\dots,a_m)$ is a multilinear expression in $\{ \alpha_{1,d(a_1)},\dots, \alpha_{m,d(a_m)} \}$.
\end{proposition}
\begin{proof}
	Since $w(x_1,\dots,x_m)$ is balanced then $w(a_1,\dots,a_m)\vert_{a_i \rightarrow 1} = 1$. Thus,  in the expansion $w(a_1,\dots,a_m) = 1 +  \sum_{k \geq 1} w_k $ with $w_k \in \coef_k$ the support of the first nonzero $w_k$ must contain at least one element in each of the sets $\{\alpha_{i,j} \mid j\geq d(a_i)\}$ $i=1,\dots,m$. This proves that  $d(w(a_1,\dots, a_m)) \geq d(a_1) + \cdots + d(a_m)$ and that the component of degree $d(a_1) + \cdots + d(a_m)$ of $w(a_1,\dots,a_m)$ is either zero or a homogeneous polynomial of degree one in each generator $\{ \alpha_{1,d(1)},\dots, \alpha_{m,d(m)} \}$.
\end{proof}

For any word $w(x_1,\dots, x_m)$ in the free loop $\fl$ we consider the deviation
\begin{displaymath}
w'_i(x_1,\dots, x_{i-1},y,z,x_{i+1},\dots, x_m) := (w\vert_{x_i \rightarrow y} w\vert_{x_{i+1} \rightarrow z})\backslash w\vert_{x_i \rightarrow yz}.
\end{displaymath}
If $w$ is balanced then  $w'_i$ is balanced too. Since the commutator $[x_1,x_2]$ and the associator $(x_1,x_2,x_3)$ are balanced then all brackets that we can obtain from them and the deviations of $(x_1,x_2,x_3)$ are     balanced too. Thus, Proposition~\ref{prop:offset} implies
\begin{corollary}
	The filtration $\Jg_1(\ring) \supset \Jg_2(\ring) \supset \cdots $ is an N-sequence.
\end{corollary}

%
\subsection{The Sabinin brackets of the N-sequence $\Jg_1(\ring) \supseteq \Jg_2(\ring) \supseteq \cdots$.}

Proposition~\ref{prop:offset} implies that modulo $\Jg_{d(a_1)+\cdots + d(a_{i-1}) + d(a_{i+1}) +\cdots + d(a_m) + d(a) + d(b)}(\ring)$
\begin{displaymath}
w(a_1,\dots,a_m)\vert_{a_i \rightarrow a \circ b} \equiv w(a_1,\dots, a_m)\vert_{a_i \rightarrow a}\circ w(a_1,\dots,a_m)\vert_{a_i \rightarrow b}
\end{displaymath}
for any balanced word $w(x_1,\dots,x_m) \in \fl$ and any $a_1,\dots, a_m, a,b \in \coef$. Even more, taking into account the multilinearity  of the component of degree $d(a_1) + \cdots + d(a_m)$ of $w(a_1,\dots,a_m)$ we observe that, using the identification (\ref{eq:identification}),  any balanced word $w(x_1, \dots, x_m)$ induces a multilinear operation 
\begin{displaymath}
\coef_{i_1} \otimes_{\field} \cdots \otimes_{\field} \coef_{i_m}  \rightarrow \coef_{i_1 + \cdots + i_m}
\end{displaymath}
and a graded multilinear operation
\begin{displaymath}
\coef \otimes_{\field} \cdots \otimes_{\field} \coef \rightarrow \coef.
\end{displaymath}
Since the deviations $P_{n,m}(x_1,\dots,x_n; y_1, \dots, y_m;z)$ and the commutator are balanced, the induced operations $p_{n,m}$ and the corresponding brackets $\langle -;-,-\rangle$ in (\ref{eq:J_brackets}) define multilinear operations on $\coef$. 

\begin{theorem}\label{prop:Sabinin_J}
	Let $\ring = \bigoplus_{i\geq 1} \ring_i$ be a graded associative algebra. Then the Sabinin brackets $\langle -;-,-\rangle$ determined by the N-sequence $\Jg_1(\ring) \supseteq \Jg_2(\ring) \supseteq \cdots $ are given by Proposition~\ref{prop:brackets}.
\end{theorem}
\begin{proof}
First, observe again that there is no loss of generality in assuming that $\ring = \coef$. We will also need some notation. Let $I:=(i_1,\dots,i_m)$, $a_1 := 1 + \alpha_{1,i_1}, \dots, a_m:= 1 + \alpha_{m,i_m}$, $a:= ((a_1 \circ a_2)\cdots )\circ a_m$, $a^*_I:=((a_1 * a_2)\cdots )* a_m$, $\alpha^*_I:= ((\alpha_{1,i_1}*\alpha_{2,i_2})\cdots)* \alpha_{m,i_m}$, $b:= 1 + \beta_j$, $c:= 1 + \gamma_k$. The notation $I' \subseteq I$ means that $I'$ is a subsequence (possibly empty) of $I$. The complementary sequence of $I'$ is denoted by $I \setminus I'$. 

The operation $p_{n,1}$ on $\coef$ induced by $P_{n,1}(x_1,\dots, x_n; y;z)$ is determined by the component of degree $i_1+ \cdots + i_n+j+k$ of $P_{n,1}(a_1,\dots, a_n; b; c)$. By Proposition~\ref{prop:offset} this component is precisely the component of  $P_{n,1}(a_1,\dots, a_n; b; c)$  of multidegree $(1,\dots, 1)$ on $\Omega' := \{\alpha_{1,i_1},\dots, \alpha_{m,i_m},\beta_j,\gamma_k \}$. We will use the symbol $\equiv_{\Omega'}$ to indicate that this component is the same for two given elements.
 
Since $P(x_1,\dots, x_n;y;z)$ is balanced with respect to $z$ we have
\begin{align*}
P_{m-1,1}(& a_1 \circ a_2, a_3,\dots, a_m; b,c) \\
& \quad = (P_{m-1,1}(a_1,a_3,\dots, a_m;b;c) \circ P_{m-1,1}(a_2,a_3,\dots, a_m;b;c))\\
& \quad\quad \circ P_{m,1}(a_1,a_2,a_3,\dots, a_m;b;c) \\
& \quad \equiv_{\Omega'} P_{m-1,1}(a_1,a_3,\dots, a_m;b;c) +  P_{m-1,1}(a_2,a_3,\dots, a_m;b;c) - 1\\
& \quad\quad + P_{m,1}(a_1,a_2,a_3,\dots, a_m;b;c) -1 \\
& \quad \equiv_{\Omega'} P_{m,1}(a_1,a_2,a_3,\dots, a_m;b;c).
\end{align*}
Iterating this observation we get
\begin{displaymath}
P_{m,1}(a_1,\dots, a_m;b;c) \equiv_{\Omega'} (a,b,c) = (a\circ(b\circ c)) \backslash ((a \circ b) \circ c).
\end{displaymath}
Since $(a \circ b) \circ c = (a \circ (b \circ c)) \circ (a,b,c)$ then
\begin{align*}
	(\alpha^*_I &* \beta_j)*\gamma_k \equiv_{\Omega'} (a \circ b) \circ c =  (a \circ (b \circ c)) \circ (a,b,c)\\
	 & \equiv_{\Omega'} \left(\left(1 + \sum_{\emptyset \neq I' \subseteq I} \alpha^*_{I'}\right) \circ (1 + \beta_j + \gamma_k + \beta_j * \gamma_k) \right) \circ (a,b,c)\\
	  & \equiv_{\Omega'} \left( 1 + \sum_{\emptyset \neq I' \subseteq I} \alpha^*_{I'} + \alpha^*_{I'}*(\beta_j * \gamma_k) + \binom{\vert I' \vert + 1}{2} \alpha^*_{I'} (\beta_j \gamma_k + \gamma_k \beta_j)\right)\circ (a,b,c).
\end{align*}
Therefore,
\begin{align*}
	(\alpha^*_I &* \beta_j)*\gamma_k - (\alpha^*_I * \gamma_k)* \beta_j - \alpha^*_I *(\beta_j* \gamma_k - \gamma_k*\beta_j) \\
	&  \equiv_{\Omega'} \left( 1 + \sum_{\emptyset \neq I' \subseteq I} \alpha^*_{I'}  \right)\circ (a,b,c) - \left( 1 + \sum_{\emptyset \neq I' \subseteq I} \alpha^*_{I'}  \right)\circ (a,c,b)\\
	&  \equiv_{\Omega'}\sum_{\substack{I' \subseteq I \\ I''= I \setminus I'}} \alpha^*_{I'}  * \left( (a_{I''},b,c) - (a_{I''},c,b) \right)  \equiv_{\Omega'} \sum_{\substack{I' \subseteq I \\ I''= I \setminus I'}} \alpha^*_{I'}  * \langle \alpha^\otimes_{I''}; \beta_j,\gamma_k \rangle
\end{align*}
where $\alpha^*_{I'} = 1$ if $I' = \emptyset$ and $\alpha^\otimes_I := \alpha_{1,i_1} \otimes \cdots \otimes \alpha_{m,i_m}$. This shows that the operations $\langle \alpha_{i_1},\dots, \alpha_{i_n}; \beta_j, \gamma_k \rangle$ satisfy the recurrence that characterizes the Shestakov-Umirbaev brackets of $(\coef^\sharp,*)$. Therefore, they are given by Proposition~\ref{prop:brackets}. Finally, it is easy to prove that the commutator $[y,z]$ induces  $b*c-c*b$. This concludes the proof.
\end{proof}

%
\section{Normal subloops}

Let $R$ be a unital associative  algebra and $\J:=\J(R) := \{t + \sum_{i\geq 1} \alpha_i t^{i+1} \mid \alpha_i \in R\}$ the loop of formal power series with coefficients in $R$ under substitution. For commutative $R$, the normal subgroups of $\J(R)$ have been studied in \cite{Kl00} and, as we will observe, the proofs remain valid in a nonassociative setting.

Let us define $\fh_n := \{ \alpha_n \in R  \mid t + \alpha_n t^{n+1} \in H \J_{n+1}\}$ for any $H \subseteq \J$.

\begin{lemma}
Let $H$ be a subloop of $\J$ with $[\J,H] \subseteq H$. Then 
\begin{enumerate}
	\item $\fh_{n+1} \subseteq \fh_{n+2}$ for all $n \geq 1$ and $\fh_1 \subseteq \fh_3$.
	\item $R\fh_nR \subseteq \fh_{n+2}$.
\end{enumerate}
\end{lemma}
\begin{proof}
Since $[\J,H] \subseteq H$ then for any  $\alpha_n \in \fh_n$, $\beta \in R$ and $i \geq 1$ we have
\begin{equation}\label{eq:absortion}
(i+1) \beta \alpha_n - (n+1) \alpha_n \beta \in \fh_{n+1}.
\end{equation}
With $i =1$ and $\beta = 1$ we obtain $(1-n) \alpha_n \in \fh_{n+1}$, thus since the characteristic of $\field$ is zero, $\fh_{n+1} \subseteq \fh_{n+2}$ for all $n \geq 1$. With $n = 1$, $i = 2$, $\beta = 1$ we get $\fh_1 \subseteq \fh_3$. Again, with $i = 1,2$ we easily get 
\begin{displaymath}
R \fh_n + \fh_n R \subseteq \fh_{n+2}
\end{displaymath}
Now, (\ref{eq:absortion}) with $\alpha_{n+1} = (i+1) \beta \alpha_n - (n+1) \alpha_n \beta \in \fh_{n+1}$, $i =1 $ and $\gamma \in R$ gives $ 2\gamma(2\beta\alpha_n - (n+1)\alpha_n\beta) - (n+2)(2 \beta \alpha_n - (n+1) \alpha_n \beta) \gamma \in \fh_{n+2}$. Thus,
\begin{equation}\label{eq:double_absortion}
2(n+1) \gamma \alpha_n \beta - 2(n+2) \beta \alpha_n \gamma \in \fh_{n+2}.
\end{equation}
By symmetry $2(n+1) \beta \alpha_n \gamma - 2 (n+2) \gamma \alpha_n \beta \in \fh_{n+2}$. Adding these two equations we get $2\gamma \alpha_n \beta - 2 \beta \alpha_n \gamma \in \fh_{n+2}$, which together with (\ref{eq:double_absortion}) gives $\beta \alpha_n \gamma \in \fh_{n+2}$, i.e. $R \fh_n R \subseteq  \fh_{n+2}$.
\end{proof}

\begin{lemma}[Lemma 3.5 in \cite{Kl00}]
Let $\alpha \in R$ and $n \geq 1$. Let $A, B \in \J$ such that
\begin{align*}
	A &\equiv t + \alpha t^{n+1} \pmod{\J_{n+1}}\\
	B & \equiv t + \alpha t^{n+2} \pmod{\J_{n+2}}.
\end{align*}
Let $\beta \in R$ and $m \geq n + 2$. Then there exist $a, c\in \J_{m-n}$ and $b, d \in \J_{m-n-1}$ such that
\begin{displaymath}
 [a,A] \circ [b,B] \equiv t + \alpha \beta t^{m+1} \quad\text{and}\quad [c,A] \circ [d,B] \equiv t + \beta \alpha t^{m+1} \pmod{\J_{m+1}}
\end{displaymath}
\end{lemma}
\begin{proof}
 Take $a= t + \lambda \beta t^{m-n+1}$ and $b = t + \mu \beta t^{m-n}$ for some $\lambda, \mu \in \field$ to be determined. Modulo $\J_{m+1}$ we have
 \begin{align*}
	  [a,A] \circ [b,B] & \equiv t + (((m-n+2) \lambda + (m-n+1)\mu)\beta \alpha \\
	  & \quad - ((n+1) \lambda + (n+2)\mu) \alpha\beta)t^{m+1}
 \end{align*}
 so we can choose adequate $\lambda, \mu$ to obtain the result in the statement.
\end{proof}

Consider $I(H)$ to be the ideal of $R$ generated by $\cup_{n \geq 1} \fh_n$ and recall that $\J(R)$ is a topological space with the topology given by considering $\{ \J_n \mid n \geq 1\}$ a base of neighborhoods of the unit element $t$.

\begin{proposition}[Proposition 1.1 in \cite{Kl00}]\label{prop:Klopsch}
	Let $R$ be a unital associative  algebra. Let $H$ be a subloop of $\J(R)$ such that $[\J(R),H] \subseteq H$ and $I(H)$ is finitely generated either as a left ideal or as a right ideal of $R$. Then $H$ is closed in $\J(R)$.
\end{proposition}
\begin{proof}
The ideal $I := I(H)$ can be written either as $I = R \sigma_1 + \cdots + R\sigma_s$ or $I = \sigma_1 R + \cdots + \sigma_s R$ for some $\sigma_1,\dots, \sigma_s \in R$. Since $R\fh_nR \subseteq \fh_{n+2}$ then there exists $n$ such that $I = \fh_n = \fh_{n+1} = \cdots$ and we can find $A_j, B_j \in H$ such that 
\begin{align*}
 A_j & \equiv t + \sigma_j t^{n+1} \pmod{\J_{n+1}},\\
 B_j & \equiv t + \sigma_j t^{n+2} \pmod{\J_{n+2}}.
\end{align*}
Since the filtration $\J_1 \supseteq \J_2 \supseteq \cdots$ is an N-sequence, the proof in \cite{Kl00} applies verbatim.
\end{proof}

\begin{proposition}
Let $R$ be a simple unital associative  algebra and let $H$ be a nontrivial subloop of $\J(R)$ such that $[\J(R),H] \subseteq H$. Then there exists $n \geq 1$ such that $\J_{n+2}(R) \subseteq H \subseteq \J_n(R)$. In particular $H$ is a closed normal subloop.
\end{proposition}
\begin{proof}
Let $n \geq 1$ be such that $H \subseteq \J_n$ but $H \not\subseteq \J_{n+1}$. Thus $\fh_n \ne 0$. Since $R\fh_n R \subseteq \fh_{n+2+i}$ ($i\geq 0$) then $\fh_{n+2+i} = R$ and $I(H) = R$, which is generated by $1$ as a left or a right $R$-module. By Proposition~\ref{prop:Klopsch} $H$ is closed. Since $R = \fh_{n+2} = \fh_{n+3} = \cdots$ then $\J_{n+2} \subseteq H$. Moreover, $[\J(R),H] \subseteq H$ and $\J_{n+2}\subseteq H$ imply that  $H/\J_{n+2}$ is a normal subloop of $\J/\J_{n+2}$, so $H$ is a normal subloop of $\J$.
\end{proof}

%
\section{Related Lie algebras}

Let $\ring = \oplus_{i \geq 0} \ring_i$ be a graded associative algebra,  $A:= \oplus_{i\geq 0} \ring_i t^{i +1}$ contained in the algebra $\ring[t]$ of polynomials in $t$ with coefficients in $\ring$ and  $\frac{d}{dt}$ the derivation of $\ring[t]$ with respect to the indeterminate $t$. $A$ is also an algebra, so $\ring[t]$ is, with the product in (\ref{eq:prod_*})
\begin{displaymath}
a*b := a'b
\end{displaymath}
where $a' := \frac{d}{dt}a$. 

The right multiplication operator by $b$ is $R^*_b = R_b \frac{d}{dt} $. Since
\begin{displaymath}
	[R^*_a,R^*_b] (c) = c''(ba-ab) + c' (b'a-a'b)
\end{displaymath}
we observe that if $\ring[\ring,\ring] = 0$ then
\begin{displaymath}
\left[R_a\frac{d}{dt},R_b\frac{d}{dt}\right] = R_{\langle a,b \rangle} \frac{d}{dt}
\end{displaymath}
where
\begin{displaymath}
\langle a, b \rangle = b'a-a'b.
\end{displaymath}
The associative algebra
\begin{displaymath}
 \phi := \alg\langle R_a \mid a \in \ring[t] \rangle \subseteq \Endo(\ring[t])
\end{displaymath}
generated by $\{R_a \mid a \in \ring[t]\}$ is commutative. Therefore, the Lie algebra $\Lie : = \{R_a \frac{d}{dt} \mid a \in S[t] \}$ is a subalgebra of $\phi \partial$ where $\partial\colon  f \mapsto [\frac{d}{dt}, f]$ is a derivation of $\phi$ and the product of $\phi\partial$ is given by $[f\partial, g\partial] := (f\partial(g) - g\partial(f))\partial$. 

Since the map $a \mapsto R_a \frac{d}{dt}$ is not injective, it is not clear at all whether or not $(\ring[t],\langle - , - \rangle)$  is Wronskian special. For pedagogical reasons we include a proof to show that this might not be the case, but the reader can skip it since we will see that these Lie algebras might fail to satisfy the standard identity $\St_5$ of degree $5$, although they satisfy $\St_6$.

\begin{proposition}\label{prop:nonWronskian}
	Let $\ring := \field e + \field v$ be the algebra determined by $e^2 = e, ev = 0, ve = v$ and $v^2 = 0$. Then the Lie algebra $(\ring[t],\langle -, -\rangle)$ is not Wronskian special.
\end{proposition}
\begin{proof}
	Assume that we have a commutative algebra $\phi$ with a derivation $f \mapsto f'$ and an injective map $\varphi \colon \ring[t] \rightarrow \phi$ such that
	\begin{equation}\label{eq:embedding}
		\varphi(b'a-a'b) = \varphi(b)'\varphi(a) - \varphi(a)' \varphi(b) 
	\end{equation}
	where we also use the notation $a'$ for the derivation $\frac{d}{dt}$ of $\ring[t]$. Define elements in $\phi$
	\begin{displaymath}
	 f_i := \varphi(et^i)\quad \text{and} \quad g_j := \varphi(vt^j).
	\end{displaymath}
	By (\ref{eq:embedding}) we have
	\begin{equation} \label{eq:commutators}
		f'_jf_i - f'_if_j = (j-i)f_{i+j-1}, \quad g'_j f_i - f'_ig_j = j g_{i+j-1} \quad \text{and} \quad g'_j g_i = g'_i g_j.
	\end{equation}
	Thus,
	\begin{align*}
		jf_1 g_{i+j-1} &= g'_jf_1f_i - f_1f'_i g_j = g'_j f_1 f_i - ((i-1)f_i - f'_1 f_i)g_j \\
		&= (g'_j f_1 - g_j f'_1)f_i - (i-1) f_i = (j-i+1)f_i g_j
	\end{align*}
	and we have
	\begin{equation} \label{eq:fi_to_f1}
		(j-i+1)f_i g_j = j f_1 g_{i+j-1}.
	\end{equation}
	With $i = j+1$ we get $j g_1 g_{2j} = 0$, i.e.
	\begin{equation}\label{eq:f1g2j=0}
		f_1 g_{2j} = 0 \quad \text{if} \quad j \geq 1.
	\end{equation}
	Taking derivatives we obtain $f_1 g'_{2j} + f'_1g_{2j} = 0$. Adding/subtracting the second equation in (\ref{eq:commutators}) we get 
	\begin{equation}\label{eq:f1g2jtog2j}
		g'_{2j}f_1 = j g_{2j} = -f'_1g_{2j} \quad \text{if} \quad j \geq 1.
	\end{equation}
	Moreover, $f'_i f_1 - f'_1f_i = (i-1)f_i$,  (\ref{eq:f1g2j=0}) and (\ref{eq:f1g2jtog2j}) imply $(i-1)f_i g_{2j} = -f'_1f_if_{2j} = j f_ig_{2j}$, so $(j-i+1)f_i g_{2j} = 0$. By (\ref{eq:fi_to_f1}) this implies $2j f_1 g_{i+2j-1} = 0$, and with $i = 0$ we get $2j f_1 g_{2j-1} = 0$. Therefore,
	\begin{equation}\label{eq:f1gj=0}
		f_1 g_j = 0 \quad \text{if} \quad j \geq 1.
	\end{equation}
	With this relation, (\ref{eq:fi_to_f1}) implies 
	\begin{equation}\label{eq:figj=0}
		f_i g_j = 0 \quad \text{if} \quad i + j \geq 2 \text{ and } i \neq j+1.
	\end{equation}
	Thus, $g'_jf_i + f'_ig_j = 0$ (we will assume the restrictions about $i,j$ in the following). Adding/subtracting this relation to the second equation in (\ref{eq:commutators}) we get 
	\begin{equation}\label{eq:g'jfi}
	2g'_j f_i = jg_{i+j-1} = -2f'_i g_j.
	\end{equation}  
	Multiplying by $2f'_k$, on the one hand for $i+j \geq 2, i \neq j+1, i+j+k \geq 3, k \neq i +j$ we obtain
	\begin{equation}\label{eq:half1}
			4 f'_k f_i g'_j = 2j f'_k g_{i+j-1} = -j (i+j-1)g_{i+j+k-2}
	\end{equation}
	but on the other hand if in addition $j+k \geq 2, k \neq j+1, i\neq j+k, i+k \neq j+2$ then we have
	\begin{align*}
		4f'_k f_i g'_j &= 4(f'_if_k + (k-i)f_{i+k-1})g'_j && \text{by \eqref{eq:commutators}} \\
		&= 4f'_i f_kg'_j + 4(k-i)f_{i+k-1}g'_j\\
		&= -j(j+k-1) g_{i+j+k-2} + 4(k-i)f_{i+k-1}g'_j && \text{by \eqref{eq:half1}}\\
		&= -j(j+k-1)g_{i+j+k-2} + 2j(k-i)g_{i+j+k-2} && \text{by \eqref{eq:g'jfi}}\\
		&= j(k-2i-j+1)g_{i+j+k-2}.
	\end{align*}
	Therefore,
	\begin{displaymath}
	j(k-i) g_{i+j+k-2} = 0
	\end{displaymath}
	which contradicts the injectivity of $\varphi$.
\end{proof}

\begin{proposition}
	The Lie algebra $(\ring[t],\langle -, -\rangle)$  in Proposition~\ref{prop:nonWronskian} does not satisfy $\St_5$.
\end{proposition}
\begin{proof}
	It is not difficult to check that $\St_5(et,et^2,vt,e,e) = 4v$. 
\end{proof}
While $\St_5$ fails in general for these Lie algebras, however we have the following result.
\begin{proposition}\label{prop:Bergman}
	Let $\ring$ be an associative algebra with $\ring[\ring,\ring] = 0$. Then the Lie algebra $(\ring[t],\langle - , - \rangle)$ satisfies $\St_6$. 
\end{proposition}
\begin{proof}
We follow the arguments by Bergman in \cite{Be79}*{Proof of Theorem 2.1}. First we expand $\St_6(x_1,x_2,x_3,x_4,x_5,z)$ in terms of associative monomials in $x_1,\dots,x_5,z$ and their iterated derivatives (only six occurrences of the derivation are allowed). Given one monomial in the expansion, say $\lambda y_1\cdots y_6$, the condition $\ring[\ring,\ring]$ ensures that we can freely reorder $y_2,\dots,y_6$. One of these $y_2,\dots, y_6$ might be $z$ or its derivatives. We focus on the remaining four factors. Since only six occurrences of the derivation are allowed, two of these four factors are of the form $\frac{d^k}{dt^k}x_i$ and $\frac{d^k}{dt^k}x_j$ for some $k$. Because of the alternating sum that defines $\St_6$, we will also have a corresponding summand $-\lambda \bar{y}_1\cdots \bar{y}_6$ where $\frac{d^k}{dt^k}x_i$  and $\frac{d^k}{dt^k}x_j$ are interchanged. Since we can freely reorder these factors, then $\lambda y_1\cdots y_6 - \lambda  \bar{y}_1\cdots \bar{y}_6$ vanishes. Thus, grouping  the summands in pairs we see that $\St_6(x_1,x_2,x_3,x_4,x_5,z) = 0$.
\end{proof}

With an extra condition we can obtain Lie algebras that satisfy $\St_5$.

\begin{proposition}\label{prop:St5}
	Let  $\ring$ be an associative algebra with $\ring[\ring,\ring] = 0$ and $[\ring,\ring]\ring^3 = 0$. Then the Lie algebra $(\ring[t],\langle - , - \rangle)$ satisfies $\St_5$. 
\end{proposition}
\begin{proof}
The conditions on $\ring$ ensure that we can freely reorder the factors in each summand of the expansion of $\St_5$ as in the proof of Proposition~\ref{prop:Bergman}, so the same proof as in the commutative case in \cite{Be79} remains valid. 
\end{proof}

It would be interesting to study whether or not the algebras in Proposition~\ref{prop:St5} are Wronskian special.

%
%
\section{Appendix: one-sided local loops} 
We observe that the formula $\alpha_m * \beta_n := (m+1) \alpha_m \beta_n$ has not right unit element in case we assume that $1$ is of degree $0$, which leads to a left loop rather than a two-sided loop.

In this section we would like to briefly discuss how to adapt the approach of Mikheev and Sabinin to the study of one-side local loops. Let $(Q,x*y,e)$ be a right (local) loop, i.e. $e$ is the right unit element and the Jacobian of the left and right multiplication operators at $e$ is nonzero, and define a two-sided loop by $$xy = x*(L^{*}_e)^{-1}y.$$ This loop is classified by its Sabinin algebra $(T_e Q, \langle - ; -,- \rangle, \Phi(-;-))$. Thus, if we include a new family of totally symmetric multilinear operation, lets say $[x_1,\dots, x_n]$ ($n \geq 1$), corresponding to the Taylor series of in normal coordinates to classify the map $L^*_e$ (see \cite{MiSa87}) then the algebraic structure 
\begin{equation}\label{eq:right_Sabinin}
(T_e Q, \langle - ; -,- \rangle, \Phi(-;-),[-])
\end{equation}
classifies the right local loop $(Q,x*y,e)$. The integration of these structures to left loops only requires the usual convergence conditions (see \cite{MiSa87}). 

From a geometrical point of view, given a right loop we define the parallel transport as
\begin{displaymath}
 \tau^e_y := dL^*_y(L^*_e)^{-1}\vert_e.
\end{displaymath}
so that
\begin{displaymath}
\tau^x_y := (dL^*_y\vert e) (dL^*_x\vert_e)^{-1}.
\end{displaymath}
This parallel transport defines a right monoalternative geodesic loop $x \times y := \exp_x \tau^e_x \exp^{-1}_e(y)$. For this loop we have
\begin{displaymath}
dL^\times_y\vert_e = \tau^e_y =  dL^*_y(L^*_e)^{-1}\vert_e.
\end{displaymath} 
We can consider the map $\Psi_x$ defined by $x*y = x \times \Psi_x(y)$. The only restrictions on $\Psi$ are
\begin{displaymath}
\Psi_x(e) = e, \quad \Psi_e(y) = e*y \quad \text{and} \quad d\Psi_x\vert_e = dL^*_e\vert_e.
\end{displaymath}
and  $(Q,*)$ is classified by $x \times y$ and $\Psi_x$. This requires the same structure as in (\ref{eq:right_Sabinin}) since we require the Taylor coefficients of $\Psi_e = L^*_e$ and those of $\Psi(x,y)$ on degrees $\geq 1$ and $\geq 2$ on the normal coordinates of $x$ and $y$ respectively. In fact, if we define $\Phi'_x:=(L^*_e)^{-1}\Psi_x$ then $xy = x (L^{*}_e)^{-1}y = x \times \Phi'_x(y)$ with 
\begin{displaymath}
\Phi'_x(e) = e, \quad \Phi'_e(y) = y \quad \text{and} \quad d\Phi'_x\vert_e = \mathrm{Id}.
\end{displaymath}
This proves that the loop $x \times y$ is also the monoalternative perturbation of the loop $xy$ and that $\Psi$ is recovered with the help of $\Phi$ and $L^*_e$.  Moreover, if the right local loop $(Q,*)$ is right monoalternative then $xy$ is a monoalternative loop. Thus $xy = x \times y$ and $x \times y = x* (L^{*}_e)^{-1}(y)$.

Therefore, the classification of local one-sided loops only requires of the usual structure of a Sabinin algebra and an extra family of totally symmetric multilinear operations $[x_1,\dots, x_n]$ ($n \geq 1$).

%
%

\end{document}